\documentclass[12pt]{amsart}

\usepackage{amsthm,amssymb,amsmath}
\theoremstyle{plain}
\newtheorem{thm}{Theorem}
\newtheorem{lem}[thm]{Lemma}
\newtheorem{prop}[thm]{Proposition}
\newtheorem{cor}[thm]{Corollary}

\theoremstyle{definition}
\newtheorem*{defn}{Definition}
\newtheorem*{rem}{Remark}
\newtheorem*{ex}{Example}

\newcommand{\newoperator}[2]{\DeclareMathOperator{#1}{#2}}
\newoperator{\spn}{span}
\newoperator{\graph}{graph}
\newoperator{\im}{Im}
\newoperator{\a.e.}{a.e.}

\usepackage[dvips]{graphicx}
\makeatletter
\DeclareRobustCommand\widecheck[1]{{\mathpalette\@widecheck{#1}}}
\def\@widecheck#1#2{%
   \setbox\z@\hbox{\m@th$#1#2$}%
   \setbox\tw@\hbox{\m@th$#1%
      \widehat{%
         \vrule\@width\z@\@height\ht\z@
         \vrule\@height\z@\@width\wd\z@}$}%
   \dp\tw@-\ht\z@
   \@tempdima\ht\z@ \advance\@tempdima2\ht\tw@ \divide\@tempdima\thr@@
   \setbox\tw@\hbox{%
      \raise\@tempdima\hbox{\scalebox{1}[-1]{\lower\@tempdima\box\tw@}}}%
   {\ooalign{\box\tw@ \cr \box\z@}}}
\makeatother

\title
{Toeplitz-composition $C^{*}$-algebras for certain finite Blaschke products}
\author{Hiroyasu Hamada}
\author{Yasuo Watatani}

\address[Hiroyasu Hamada]{Graduate School of Mathematics, Kyushu University,
Hakozaki, Fukuoka, 812-8581, Japan.}
\email{h-hamada@math.kyushu-u.ac.jp}

\address[Yasuo Watatani]{Department of Mathematical Sciences, 
Kyushu University, Hakozaki, Fukuoka, 812-8581, Japan.}
\email{watatani@math.kyushu-u.ac.jp}

\begin{document}

\begin{abstract}
Let $R$ be a finite Blaschke product of degree at least two with 
$R(0)=0$. Then there exists a relation between the associated composition 
operator $C_R$ on the Hardy space and the $C^*$-algebra $\mathcal{O}_R (J_R)$
associated with the complex dynamical system $(R^{\circ n})_n$ 
on the Julia set $J_R$.  We study the $C^*$-algebra $\mathcal{TC}_R$ 
generated by both the composition operator $C_R$ and the Toeplitz operator 
$T_z$ to show that the quotient algebra by the ideal of the compact 
operators is isomorphic to the $C^*$-algebra $\mathcal{O}_R (J_R)$, which is 
simple and purely infinite.

\medskip\par\noindent
KEYWORDS: composition opearator, Blaschke product, Toeplitz operator, 
C*-algebra, complex dynamical system

\medskip\par\noindent
AMS SUBJECT CLASSIFICATION: 46L55, 47B33, 37F10, 46L08

\end{abstract}

\maketitle

\section{Introduction}

Let $\mathbb{D}$ be the open unit disk in the complex plane and 
$H^2 (\mathbb{D})$ the Hardy (Hilbert) space of 
analytic functions whose power series have square-summable 
coefficients. 
For an analytic self-map
$\varphi:\mathbb{D} \rightarrow \mathbb{D} $,  
the composition operator 
$C_{\varphi} :H^2 (\mathbb{D}) \rightarrow H^2 (\mathbb{D})$ 
is defined by $C_{\varphi}(g) = g \circ \varphi$ for $g \in H^2 (\mathbb{D})$ 
and known to be a bounded operator by the Littlewood subordination theorem 
\cite{Littlewood}. 
The study of composition operators 
on the Hardy space $H^2 (\mathbb{D})$ gives a fruitful interplay 
between complex analysis and operator theory as 
shown, for example,  in the books of Shapiro \cite{Shapiro}, 
Cowen and MacCluer \cite{Cowen and MacCluer} and
Mart\'{i}nez-Avenda\~{n}o and Rosenthal \cite{Rosenthal}. 
Since the work \cite{Cowen} by Cowen, good representations of 
adjoints of composition 
operators have been investigated.  Consult 
Cowen and Gallardo-Guti\'{e}rrez 
\cite{CG}, Mart\'{i}n and Vukoti\'{e} \cite{MV}, 
Hammond, Moorhouse and Robbins  \cite{HMR}, 
Bourdon and Shapiro \cite{BS} to know recent achievement 
for  adjoints of composition operators with rational symbols. 
In this paper, we only need an old result by McDonald in \cite{McDonald} 
for finite Blashschke products.   

On the other hand,  
for a branched covering $\pi : M \rightarrow M$, 
Deaconu and Muhly \cite{Deaconu and Muhly} introduced a $C^*$-algebra 
$C^*(M,\pi)$ as the $C^*$-algebra of the r-discrete 
groupoid constructed by Renault \cite{Renault}. In particular, 
they study rational functions on the Riemann sphere $\hat{\mathbb C}$.  
Iterations $(R^{\circ n})_n$ of $R$ by composition give complex dynamical 
systems. 
In \cite{Kajiwara and Watatani} 
Kajiwara and the second-named author introduced slightly different 
$C^*$-algebras 
${\mathcal O}_R(\hat{\mathbb C})$,  
${\mathcal O}_R(J_R)$ and ${\mathcal O}_R(F_R)$
associated with complex dynamical system $(R^{\circ n})_n$ 
on the Riemann sphere $\hat{\mathbb C}$,  
the Julia set $J_R$ and the Fatou set $F_R$ of $R$. 
The $C^*$-algebra ${\mathcal O}_R(J_R)$ is defined  
as a Cuntz-Pimsner algebra \cite{Pimsner} of a Hilbert bimodule, 
called $C^*$-correspondence, $
C(\graph R|_{J_R})$ over $C(J_R)$. 
We regard the algebra ${\mathcal O}_R(J_R)$ as 
a certain analog of the 
crossed product $C(\Lambda_{\Gamma}) \rtimes \Gamma$ of 
$C(\Lambda_{\Gamma})$ by a boundary action of a Kleinian group $\Gamma$ 
on the limit set $\Lambda_{\Gamma}$.  

The aim of the paper is to show that 
there exists a relation between composition 
operators on the Hardy space and the $C^*$-algebras $\mathcal{O}_R (J_R)$
associated with the complex dynamical systems $(R^{\circ n})_n$ 
on the Julia sets $J_R$.  

We recall that the $C^*$-algebra $\mathcal{T}$ generated by 
the Toeplitz operator $T_z$ contains all continuous symbol Toeplitz
operators, and  its quotient by the ideal 
of the compact operators  
on $H^2 (\mathbb{T})$ is isomorphic to the commutative $C^*$-algebra 
$C(\mathbb{T})$ of all continuous functions on $\mathbb{T}$. 
For an analytic self-map
$\varphi:\mathbb{D} \rightarrow \mathbb{D} $,  
we denote by 
$\mathcal{TC}_\varphi$ the Toeplitz-compostion $C^*$-algebra
generated by both the composition operator $C_{\varphi}$ and the Toeplitz 
operator $T_z$. Its quotient algebra  by the ideal 
${\mathcal K}$  of the compact operators is denoted by 
$\mathcal{OC} _{\varphi}$. 
Recently Kriete, MacCluer and Moorhouse
\cite{Kriete MacCluer Moorhouse 1, Kriete MacCluer Moorhouse 2}
studied the Toeplitz-compostion $C^*$-algebra $\mathcal{TC}_\varphi$
for a certain linear fractional self-map $\varphi$.  
They describe the quotient $C^*$-algebra $\mathcal{OC} _{\varphi}$
concretely as a subalgebra of $C(\Lambda) \otimes M_2(\mathbb C)$. 
If $\varphi(z) = e^{-2\pi i\theta}z$ for some irrational 
number $\theta$, then the Toelitz-composion $C^*$-algebra 
$\mathcal{TC}_{\varphi}$ is an extension of the irrational rotation algebra
$A_{\theta}$ by ${\mathcal K}$ and studied in Park \cite{Park}. 
Jury \cite{Jury 1, Jury 2} investigated the $C^*$-algebra generated by 
a group of composition operators with the symbols belonging to a 
non-elementary Fuchsian group $\Gamma$ to relate it   with 
extensions of the crossed product $C(\mathbb{T}) \rtimes \Gamma$
by ${\mathcal K}$. 

In this paper we study a certain class in the remaining cases consisting of 
finite Blaschke products $R$ of degree $n \geq 2$ with $R(0)=0$. 
The boundary 
$\mathbb{T}$ of the open unit 
disk $\mathbb{D}$ is the Julia set $J_R$ of the Blaschke product $R$. 
We show that the quotient algebra  $\mathcal{OC} _{R}$ of the 
Toelitz-composion $C^*$-algebra 
$\mathcal{TC}_R$ by the ideal ${\mathcal K}$ is isomorphic to the $C^*$-algebra $\mathcal{O}_R (J_R)$ 
associated with the complex dynamical system $(R^{\circ n})_n$, 
which is simple and purely infinite. 
We should remark that 
the notion of transfer operator by Exel in \cite{Exel} is one of 
the keys to clarify the above relation. In fact the corresponding 
operator of the composition operator in the quotient algebra 
is the implementing isometry operator. 

The Toeplitz-compostion $C^*$-algebra 
depends on the analytic structure of the Hardy space by the costruction. 
The finite  Blaschke product $R$ is not conjugate with $z^n$  by any 
M\"{o}bius automorphism unless $R(z) = \lambda z^n$. But we can show that 
the quotient algebra  $\mathcal{OC} _{R}$ is isomorphic to $\mathcal{OC} _{z^n}$as a corollary of our main theorem. This enable us to compute 
$K_0(\mathcal{OC} _{R})$ and $K_1(\mathcal{OC} _{R})$ easily.

\section{Toeplitz-composition $C^*$-algebras}

Let $L^2(\mathbb{T})$ denote the square integrable measurable functions on
$\mathbb{T}$ with respect to the nomalized Lebesgue measure.
The Hardy space $H^2(\mathbb{T})$ is the closed subspace of 
$L^2(\mathbb{T})$ consisting of the functions whose negative 
Fourier coefficients vanish. We put $H^{\infty} (\mathbb{T}) :=
H^2(\mathbb{T}) \cap L^{\infty} (\mathbb{T})$. 

The Hardy space $H^2(\mathbb{D})$ is the Hilbert space consisting
of all analytic  functions $g(z) = \sum_{k=0} ^{\infty} c_k z^ k$
on the open unit disk $\mathbb{D}$ such that
$\sum_{k=0} ^{\infty} | c_k | ^2 < \infty$.
The inner product is given  by
\[
 (g|h) = \sum_{k=0} ^{\infty} c_k \overline{d_k}
\]
for
$g(z)= \sum_{k=0} ^{\infty} c_k z^k$ and
$h(z) = \sum_{k=0} ^{\infty} d_k z^k$.

We identify $H^2 (\mathbb{D})$ with $H^2(\mathbb{T})$ by a unitary 
$U : H^2 (\mathbb{D}) \rightarrow H^2(\mathbb{T})$. We note that 
$\widetilde{g} = Ug$ is given as
\[
 \widetilde{g}(e^{i \theta}) :=\lim_{r \to 1-} g(re^{i \theta})  \ \
 \a.e. \theta
\]
for $g \in H^2 (\mathbb{D})$ by Fatou's theorem.
Moreover the inverse $\widecheck{f} = U^* f$ is given as a Poisson integral 
\[
 \widecheck{f}(re^{i\theta}) := \frac{1}{2 \pi} \int_{0} ^{2 \pi} P_r (\theta -
 t)f(e^{it}) dt
\]
for $f \in H^2 (\mathbb{T})$, where $P_r$ is the Poisson kernel defined by
\[
 P_r (\theta) = \frac{1-r^2}{1-2r\cos{\theta}+r^2}, 
 \quad 0 \leq r < 1, \ 0 \leq \theta \leq 2 \pi.
\]

Let  $P_{H^2} : L^2 (\mathbb{T}) \rightarrow H^2 (\mathbb{T}) \subset L^2 (\mathbb{T})$ be the projection.
For $a \in L^{\infty}(\mathbb{T})$,  the Toeplitz operator $T_a$ on
$H^2(\mathbb{T})$ is defined by $T_a f = P_{H^2} a f$
for $f \in H^2 (\mathbb{T})$.

Let $\varphi: \mathbb{D} \rightarrow \mathbb{D}$ be an analytic 
self-map. 
Then the  composition operator $C_{\varphi}$ on $H^2 (\mathbb{D})$ is defined by
$C_{\varphi} g = g \circ \varphi$ for $g \in H^2 ({\mathbb{D}})$. By the Littlewood subordination
theorem, $C_{\varphi}$ is always bounded.  

We can regard Toeplitz operators and composition operators as acting 
 on the same Hilbert space  by the unitary $U$ above.
More precisely, we put  $\widecheck{T_a} = U^* T_a U $ and
$\widetilde{C_{\varphi}} = U C_{\varphi} U^*$. 
If $\varphi$ is a inner function, then 
we know  that 
$\widetilde{g \circ \varphi} =
\widetilde{g} \circ \widetilde{\varphi}$  for  $g \in H^2 (\mathbb{D})$ 
by Ryff \cite[Theorem 2]{Ryff}. 
Therefore we may write 
$\widetilde{C_{\varphi}}f = f \circ \widetilde{\varphi}$ 
for $f \in H^2(\mathbb{T})$.

\begin{defn} 
For an analytic self-map
$\varphi:\mathbb{D} \rightarrow \mathbb{D} $, 
we denote by $\mathcal{TC} _{\varphi}$ the $C^*$-algebra generated
by the Toeplitz operator $\widecheck{T_z}$ and 
the composition operator  $C_{\varphi}$ on $H^2(\mathbb{D})$. 
The $C^*$-algebra $\mathcal{TC} _{\varphi}$ is 
called the Toeplitz-composition $C^*$-algebra with symbol $\varphi$. 
Since $\mathcal{TC} _{\varphi}$ contains the ideal 
$K(H^2(\mathbb{D}))$ of compact operators,  
we define a $C^*$-algebra $\mathcal{OC} _{\varphi}$ to be the quotient
$C^*$-algebra $\mathcal{TC} _{\varphi} / K(H^2(\mathbb{D}))$.
By the unitary $U:H^2 (\mathbb{D}) \rightarrow H^2 (\mathbb{T})$ above, 
we usually identify  $\mathcal{TC}_{\varphi}$  with 
the $C^*$-algebra generated by $T_z$ and $\widetilde{C_\varphi}$. We also  
use the same notations  $\mathcal{TC}_{\varphi}$
and $\mathcal{OC} _{\varphi}$ for the operators on $H^2(\mathbb{T})$. 
But we sometimes need to treat them carefully. Therefore  
we often use the notations 
$\widetilde{g} = Ug$ and $\widecheck{f} = U^* f$ to avoid confusions 
in the paper. Wise readers may neglect these troublesome notations. 
\end{defn}

\section{$C^*$-algebras associated with complex dynamical systems}

We recall the construction of Cuntz-Pimsner algebras \cite{Pimsner}. 
Let $A$ be a $C^*$-algebra and $X$ be a Hilbert right $A$-module.  
We denote by $L(X)$ the algebra of the adjointable bounded operators 
on $X$.  
For $\xi$, $\eta \in X$, the operator $\theta _{\xi,\eta}$
is defined by $\theta _{\xi,\eta}(\zeta) = \xi(\eta|\zeta)_A$
for $\zeta \in X$. 
The closure of the linear span of these operators is denoted by $K(X)$. 
We say that 
$X$ is a Hilbert $C^*$-bimodule (or $C^*$-correspondence) 
over $A$ if $X$ is a Hilbert right $A$-module 
with a *-homomorphism $\phi : A \rightarrow L(X)$. We always assume 
that $X$ is full and $\phi$ is injective. 
   Let $F(X) = \bigoplus _{n=0}^{\infty} X^{\otimes n}$
be the full Fock module of $X$ with a convention $X^{\otimes 0} = A$. 
 For $\xi \in X$, the creation operator $T_{\xi} \in L(F(X))$ is defined by 
\[
T_{\xi}(a) =  \xi a  \quad \text{and} \quad 
T_{\xi}(\xi _1 \otimes \dots \otimes \xi _n) = \xi \otimes 
\xi _1 \otimes \dots \otimes \xi _n .
\]
We define $i_{F(X)}: A \rightarrow L(F(X))$ by 
$$
i_{F(X)}(a)(b) = ab \quad \text{and} \quad 
i_{F(X)}(a)(\xi _1 \otimes \dots \otimes \xi _n) = \phi (a)
\xi _1 \otimes \dots \otimes \xi _n 
$$
for $a,b \in A$.  The Cuntz-Toeplitz algebra ${\mathcal T}_X$ 
is the C${}^*$-algebra acting on $F(X)$ generated by $i_{F(X)}(a)$
with $a \in A$ and $T_{\xi}$ with $\xi \in X$. 

Let $j_K | K(X) \rightarrow {\mathcal T}_X$ be the homomorphism 
defined by $j_K(\theta _{\xi,\eta}) = T_{\xi}T_{\eta}^*$. 
We consider the ideal $I_X := \phi ^{-1}(K(X))$ of $A$. 
Let ${\mathcal J}_X$ be the ideal of ${\mathcal T}_X$ generated 
by $\{ i_{F(X)}(a) - (j_K \circ \phi)(a) \, ; \, a \in I_X\}$.  Then 
the Cuntz-Pimsner algebra ${\mathcal O}_X$ is defined as 
the quotient ${\mathcal T}_X/{\mathcal J}_X$ . 
Let $\pi : {\mathcal T}_X \rightarrow {\mathcal O}_X$ be the 
quotient map.  
We set $S_{\xi} = \pi (T_{\xi})$ and $i(a) = \pi (i_{F(X)}(a))$. 
Let $i_K : K(X) \rightarrow {\mathcal O}_X$ be the homomorphism 
defined by $i_K(\theta _{\xi,\eta}) = S_{\xi}S_{\eta}^*$. Then 
$\pi((j_K \circ \phi)(a)) = (i_K \circ \phi)(a)$ for $a \in I_X$.   

The Cuntz-Pimsner algebra ${\mathcal O}_X$ is 
the universal C${}^*$-algebra generated by $i(a)$ with $a \in A$ and 
$S_{\xi}$ with $\xi \in X$  satisfying that 
$i(a)S_{\xi} = S_{\phi (a)\xi}$, $S_{\xi}i(a) = S_{\xi a}$, 
$S_{\xi}^*S_{\eta} = i((\xi | \eta)_A)$ for $a \in A$, 
$\xi, \eta \in X$ and $i(a) = (i_K \circ \phi)(a)$ for $a \in I_X$.
We usually identify $i(a)$ with $a$ in $A$.  If $A$ is unital and 
$X$ has a finite basis $\{u_i\}_{i=1}^n$ in the sense that 
$\xi = \sum _{i=1}^n u_i(u_i | \xi)_A$, then the last condition can be 
replaced by that there exist a finite set $\{v_j\}_{j=1}^m \subset X$ 
such that $\sum _{j=1}^m S_{v_j}S_{v_j}^* = I$. Then   
$\{v_j\}_{j=1}^m$  becomes a finite basis of $X$ automatically.

Next we introduce the $C^*$-algebras associated with complex dynamical 
systems as in \cite{Kajiwara and Watatani}. 
Let $R$ be a rational function of degree at least two. 
The sequence $(R^{\circ n})_n$ of iterations of composition by $R$ 
gives a complex dynamical system on the Riemann sphere
 $\hat{\mathbb C} = {\mathbb C} 
\cup \{ \infty \}$. 
The Fatou set $F_R$ of $R$ is the maximal open subset of 
$\hat{\mathbb C}$ on which $(R^{\circ n})_n$ is equicontinuous (or 
a normal family), and the Julia set $J_R$ of $R$ is the 
complement of the Fatou set in $\hat{\mathbb C}$.  
We denote by $e(z_0)$  the 
branch index of $R$ at $z_0$.  
Let $A= C(\hat{\mathbb C})$ and $X = C(\graph R)$ be the set of continuous 
functions on $\hat{\mathbb C}$ and $\graph R$ respectively, 
where $\graph R = \{(x,y) \in \hat{\mathbb C}^2 \, ; \, y = R(x)\} $ 
is the graph of $R$.
Then $X$ is an $A$-$A$ bimodule by 
$$
(a\cdot \xi \cdot b)(x,y) = a(x)\xi(x,y)b(y),\quad a,b \in A,\; 
\xi \in X.$$ 
We define an $A$-valued inner product $(\ |\ )_A$ on $X$ by 
$$
(\xi|\eta)_A(y) = \sum _{x \in R^{-1}(y)} e(x) \overline{\xi(x,y)}\eta(x,y),
\quad \xi,\eta \in X,\; y \in \hat{\mathbb C}.$$  
Thanks to the branch index $e(x)$, the inner product above gives a continuous  
function and $X$ is a full Hilbert bimodule over $A$ without completion. 
The left action of $A$ is unital and faithful. 

Since the Julia set $J_R$ is completely invariant under $R$, i.e., 
$R(J_R) = J_R = R^{-1}(J_R)$, we can consider the restriction 
$R|_{J_R} : J_R \rightarrow J_R$, which will be often denoted by 
the same letter $R$.
 Let $\graph R|_{J_R} = \{(x,y) \in J_R \times J_R \, ; \, y = R(x)\} $
be the graph of the restriction map $R|_{J_R}$ and 
$X_R = C(\graph R|_{J_R})$. 
In the same way as above, $X_R$ is a full Hilbert bimodule over $C(J_R)$.

\begin{defn} 
The $C^*$-algebra ${\mathcal O}_R(\hat{\mathbb C})$ 
on $\hat{\mathbb C}$ is defined 
as the Cuntz-Pimsner algebra of the Hilbert bimodule 
$X= C(\graph R)$ over 
$A = C(\hat{\mathbb C})$. 
We also define the $C^*$-algebra ${\mathcal O}_R(J_R)$ 
on the Julia set $J_R$ 
as the Cuntz-Pimsner algebra of the Hilbert bimodule $
X_R = C(\graph R|_{J_R})$ over $A = C(J_R)$.
\end{defn}

\section{transfer operators and adjoints of composition operators}

In the rest of the paper, we assume that $R$ is a
a finite Blaschke product of degree at least two with $R(0)=0$,
that is
\[
 R(z) = \lambda z \prod_{k=1} ^{n-1} \frac{z-z_k}{1-\overline{z_k}z}
      = \lambda \ \prod_{k=0} ^{n-1} \frac{z-z_k}{1-\overline{z_k}z},
      \quad z \in \hat{\mathbb{C}},
\]
where $n \geq 2, \ z_1,\dots, z_{n-1} \in \mathbb{D}, \  |\lambda | = 1 $
and $z_0 = 0$.  
Thus $R$ is a rational function with degree $\deg R = n$. 
Since $R$ is an inner function,  $R$ is an analytic self-map on $\mathbb{D}$.
We consider the composition operator $C_R$ with symbol $R$. Since $R$ is inner 
and  $R(0) = 0$,  
$C_R$ is an isometry by Nordgren \cite{Nordgren}.  
We note the following fact:

\begin{align}
 \frac{z R'(z)}{R(z)} = 1 + \sum_{k=1}^{n-1}
 \frac{1-|z_k|^2}{|z-z_k|^2} >0, \quad z \in
 \mathbb{T}. \label{derivative}
\end{align}
Thus $R$ has no branched  points on  $\mathbb{T}$ and the branch index 
$e(z)=1$ for any $z \in \mathbb{T}$. 
Furthermore the Julia set $J_R$ of $R$ is $\mathbb{T}$ (see,
for example, \cite[Page 70-71]{Milnor}) and it coincides with 
the boundary of the disk  $\mathbb{D}$. 

Let $A=C(\mathbb{T})$ and $h \in A$ be a positive invertible element.
Set $\graph R|_{\mathbb{T}} = \{ (z, w) \in \mathbb{T}^2 \, ; \, w = R(z) \}$
and $X_{R,h}=C(\graph R|_{\mathbb{T}})$.
We define 
\begin{align*}
(a \cdot \xi \cdot b)(z,w) &= a(z)\xi(z,w)b(w), \\
(\xi |\eta )_{A, h} (w) &= \sum _{z \in R^{-1}(w)}
h(z) \overline{\xi (z,w)} \eta (z,w)
\end{align*}
for $a,b \in A $ and $\xi, \eta \in X_{R,h}$.
We  see that $X_{R,h}$ is a pre-Hilbert $A$-$A$ bimodule
whose left action is faithful. Since 
the Julia set $J_R$ = $\mathbb{T}$ has no branched point, for the 
constant function $h = 1$,  $X_{R,1}$ has a finite basis and 
coincides with the Hilbert $A$-$A$
bimodule $X_R= C(\graph R|_{\mathbb{T}})$. 
Let $\{u_k \}_{k=1}^N$ be a basis of $X_R$. For a positive 
invertible element $h \in A$, put $v_k = h^{-1/2} u_k$.  Then $ \{ v_k \}_{k=1}^N $ is a basis of $X_{R,h}$,  and $X_{R,h}$ is also a Hilbert module without 
completion.

\begin{defn}
The $C^*$-algebra $\mathcal{O}_{R,h} (\mathbb{T})$ is definded as the 
Cuntz-Pimsner algebra of the Hilbert bimodule
$X_{R,h} = C(\graph R|_{\mathbb{T}})$ over $A=C(\mathbb{T})$.
The $C^*$-algebra $\mathcal{O}_{R,h} (\mathbb{T})$ is the universal $C^*$-algebra
generated  by $\{ \hat{S}_{\xi} \ ; \ \xi \in X_{R,h} \}$ and $A$
satisfying the following relations:
\[
 a \hat{S}_{\xi} = \hat{S}_{a \cdot \xi}, \ \hat{S}_{\xi} b =
 \hat{S}_{\xi \cdot b}, \ \hat{S}_{\xi} ^* \hat{S}_{\eta} = (\xi|
 \eta )_{A ,h}, \  \sum_{k=1} ^N \hat{S}_{v_k}
 \hat{S}_{v_k} ^*
 = I
\]
for $a,b \in A$ and $\xi, \eta \in X_{R,h}$. 
The $C^*$-algebra $\mathcal{O}_{R,h} (\mathbb{T})$ is in fact 
a topological quiver algebra in  Muhly and Solel \cite{MS},
Muhly and Tomforde \cite{MT}. 
\end{defn}

We will use  the symbol $a$ and $S_{\xi}$ to denote the generator of
$\mathcal{O}_R (J_R)$ for $a \in A$ and $\xi \in X_R$.

In the rest of the paper, we choose and fix $h \in A$ defined by 
\[
 h(z)= \frac{nR(z)}{z R'(z)}, \quad z \in \mathbb{T}. 
\]
Then $h$ is positive and invertible by (\ref{derivative}).

We need and collect several facts as lemmas to prove our main theorem.  
Some of them might be considered folklore. 

\begin{lem} \label{lem:isomorphic}
Let $R$ be a finite Blaschke product of degree at least two with $R(0)=0$. 
Then the $C^*$-algebra $\mathcal{O}_{R,h} (\mathbb{T})$ is isomorphic to the
$C^*$-algebra $\mathcal{O}_R (J_R)$ by an isomorphism $\Phi$
such that $\Phi(a) = a$ and $\Phi(\hat{S}_{\xi}) = h^{1/2} S_{\xi}$
for $a \in A, \, \xi \in X_{R,h}$.
\end{lem}

\begin{proof}
Let $V_{\xi} = h^{1/2} S_{\xi}$ for $\xi \in X_{R,h}$.  
Then we have that 
\[
 a V_{\xi} = V_{a \cdot \xi}, \ V_{\xi} b = V_{\xi \cdot b}, 
\]
\[
 \ V_{\xi} ^* V_{\eta} = ( h^{1/2} \cdot {\xi} | h^{1/2} \cdot {\eta} )_A 
=  ( \xi| \eta )_{A ,h} , \ \ 
 \sum_{k=1} ^N V_{v_k} V_{v_k} ^* = I
\]
for $a,b \in A$ and $\xi, \eta \in X_{R,h}$.
By the universality, we have the desired isomorphism. 
\end{proof}

\begin{defn}
For a function $f$ on  $\mathbb{T}$,
we define a function $\mathcal{L}_R(f)$ on $\mathbb{T}$ by
\[
 \mathcal{L}_R (f)(w) = \frac{1}{n}\sum_{z \in R^{-1} (w)} h(z) f(z) 
= \sum_{z \in R^{-1} (w)} \frac{R(z)}{z R'(z)} f(z),
\quad \ w \in \mathbb{T}.
\]
\end{defn}

We mainly consider the restrictions of $\mathcal{L}_R$ to $A = C(\mathbb{T})$ 
and $H^2(\mathbb{T})$ and use the same notation if no confusion can arise.

The notion of transfer operator by Exel in \cite{Exel} is one of 
the keys to clarify our situation. 

\begin{lem} \label{lem:L(1)=1}
Let $R$ be a finite Blaschke product of degree at least two with $R(0)=0$. 
Let $A=C(\mathbb{T})$ and $\alpha : A \rightarrow A$ 
be the unital *-endomorphism defined by $(\alpha (a))(z)= a(R(z))$ for
$a \in A$ and $z \in \mathbb{T}$.
Then the restriction of $\mathcal{L}_R$ to $A=C(\mathbb{T})$ is  a transfer 
operator for the pair $(A, \alpha)$ in the sense of Exel, that is, 
$\mathcal{L}_R$ is a positive linear map such that
\[
\mathcal{L}_R(\alpha(a) b) = a \mathcal{L}_R(b) \text{ for } a,b \in A. 
\]
Moreover $\mathcal{L}_R$ satisfies that $\mathcal{L}_R(1) = 1$. 
\end{lem}

\begin{proof}
The only non-trivial one is to show that  $\mathcal{L}_R(1) = 1$. 
But this is an easy calculus as follows:
We fix $w \in \mathbb{T}$ and let $\alpha_1, \dots ,\alpha_n$ be the
exactly different solutions of $R(z)=w$.
By the partial fraction decompostion, there exists $\lambda_1, \dots
,\lambda_n \in \mathbb{C}$ such that
\[
  \frac{R(z)/z}{R(z)-w}= \sum_{k=1} ^n \frac{\lambda_k}{z-\alpha_k}.
\]
Multiplying by $z$ and letting  $z \rightarrow \infty $, we have
$\sum_{k=1} ^n \lambda_k =1. \label{Blaschke2}$ 
Moreover multiplying by $z-\alpha_l$ and letting 
$z \rightarrow \alpha_l$, we get 
$\lambda_l = \frac{R(\alpha_l)}{\alpha_l R'(\alpha_l)}. $
This shows that $\mathcal{L}_R(1) = 1$.  See, for example, 
\cite{Daepp Gorkin Mortini}.

\end{proof}

Let $R$ be a finite Blaschke product of degree at least two with  $R(0)=0$. 
J. N. McDonald \cite{McDonald} calculated the adjoint of $C_R$ and gave 
a formula. We follow some of his argument, but we also need a different  
formula to prove our main theorem. 
We claim that there exist $\theta_0 \in [0, 2 \pi]$ and a strictly increasing
continuously differentiable function
$\psi:[\theta_0 - 2 \pi \, , \, \theta_0] \rightarrow \mathbb{R}$
such that $\psi(\theta_0 - 2 \pi)=0$, $\psi(\theta_0)=2n \pi$
and $R(e^{i \theta}) = e^{i \psi(\theta)}$.

\begin{lem} \label{lem:bounded}
Let $R$ be a finite Blaschke product of degree at least two with $R(0)=0$. 
Then $\mathcal{L}_R(H^2(\mathbb{T})) \subset H^2(\mathbb{T})$ and 
the restriction 
$\mathcal{L}_R:H^2(\mathbb{T}) \rightarrow H^2(\mathbb{T})$ is a bounded
operator such that $\widetilde{C_R} ^* = \mathcal{L}_R$.
\end{lem}

\begin{proof}
Let $\psi$ be the map defined above.
Set $t=\psi(\theta)$ and $\sigma_k (t) = \psi^{-1} (t+2(k-1) \pi)$
for $1 \leq k \leq n$ and $0 \leq t \leq 2 \pi$.
If we differentiate $R(e^{i \sigma_k(t)})=e^{it}$ with respect to $t$,
then
\[
 \sigma_k ' (t) =   \frac{R(e^{i \sigma_k(t)})}{R' (e^{i \sigma_k(t)}) e^{i
 \sigma_k(t)}}.
\]
Therefore for $\xi_l(z):= z^l$, 
we have that

\begin{align*}
(C_R \xi_l | f)
 &= \frac{1}{2\pi} \int_{0} ^{2 \pi} R(e^{i \theta})^l \overline{f(e^{i \theta})} d\theta
  = \frac{1}{2\pi} \int_{\theta_0- 2 \pi} ^{\theta_0}  R(e^{i \theta})^l \overline{f(e^{i
      \theta})} d\theta \\
 &= \frac{1}{2\pi} \int_{0} ^{2 n \pi} e^{ilt}
    \overline{f(e^{i\psi^{-1}(t)}) (\psi^{-1}(t))'} dt \\
 &= \frac{1}{2\pi} \sum_{k=1} ^n \int_0 ^{2 \pi} e^{ilt}
    \overline{f(e^{i \sigma_k(t)}) \sigma_k '(t)} dt \\
 &= \frac{1}{2\pi} \int_0 ^{2 \pi} e^{ilt}
    \overline{\mathcal{L}_R(f)(e^{it})} dt
  = (\xi_l | \mathcal{L}_R(f) )
\end{align*}
for $f \in H^2(\mathbb{T})$ and $l \geq 0$.
Thus $\widetilde{C_R} ^* = \mathcal{L}_R$ and $\mathcal{L}_R$ is bounded.
\end{proof}

We need an appropriate basis of $H^2 (\mathbb{T})$ for $R$. 
Let
\begin{align*}
 e_l (z) = \begin{cases}
            \displaystyle \frac{\sqrt{1- | \beta_0 |^2}}{1-\overline{\beta_0}z} &
            \text{$l=0$}  \\[10pt]
            \displaystyle \frac{ \alpha_l \sqrt{1- | \beta_l | ^2}}{1- \overline{\beta_l} z} \prod
             _{k=0} ^{l-1} \dfrac{z- \beta_k}{1- \overline{\beta_k}z} &
            \text{$l \geq 1$} \\
           \end{cases}
\end{align*}
where
$\alpha_{kn+l}= \lambda^k, \, \beta_{kn+l} = z_l$
for $k \geq 0$ and $0 \leq l \leq n-1$.
Since $\sum_{k=0} ^{\infty} (1-|\beta_k |) = \infty$,
$\{ e_k \}_{k=0} ^{\infty}$ is an orthonormal basis 
of $H^2 (\mathbb{T})$ as in Ninness, Hjalmarsson and Gustafsson
\cite[Theorem 2.1]{NHG}, 
Ninness and Gustafsson \cite[Theorem 1]{Ninness and Gustafsson}.  
We write
\begin{align*}
 Q_l (z) &= \frac{\sqrt{1- | z_l |^2 }}{1-\overline{z_l} z}, \\
 R_l (z) &= \begin{cases}
            1 & \text{$l=0$}  \\[5pt]
            \displaystyle \prod _{k=0} ^{l-1} \frac{z-z_k}{1- \overline{z_k} z} &
             \text{$l \geq 1$}. \\
            \end{cases}
\end{align*}
Thus
$e_{kn+l}=Q_l R_l R^k$ for $k \geq 0$ and $0 \leq l \leq  n-1$, where
$R^k$ is the $k$-th power of $R$ with respect to pointwise 
multiplication. 

\begin{prop} \label{prop:transfer}
Let $R$ be a finite Blaschke product of degree at least two with $R(0)=0$. 
Then for any $a \in C(\mathbb{T})$, we have 
\[
\widetilde{C_R} ^* T_a \widetilde{C_R} = T_{\mathcal{L}_R(a)}.
\]
\end{prop}

\begin{proof}
We first examine the case that $a(z) = z^j$ for $j \geq 0$.
Consider the $L^2$-expansion of $a$ by the basis $\{ e_l \}_{l=0} ^{\infty}$:
\[
 a = \sum_{l=0} ^{\infty} c_l R^l + g,
\]
where
$g \in (\im \widetilde{C_R}) ^{\perp} \cap H^2 (\mathbb{T})$.
For $\xi_m(z) = z^m$ with $m \geq 0$, we have
\[
 (T_a \widetilde{C_R} \xi_m)(z) = a(z) R(z)^m
 = \sum_{l=0} ^{\infty} c_l R(z)^{l+m} + g(z)R(z)^m .
\]
It is clear that
\begin{align*}
\im \widetilde{C_R} &= \overline{\spn} \{e_{kn} \, ; \,  k \geq 0 \} \\
\intertext{and}
(\im \widetilde{C_R})^{\perp} \cap H^2 (\mathbb{T}) &= \overline{\spn}
 \{e_{kn+l} \, ; \, k \geq 0, \, 1 \leq l \leq n-1  \},
\end{align*}
where $\overline{\spn}$ means the closure of linear span. 
Therefore 
$gR^m$ is also in $(\im \widetilde{C_R})^{\perp} \cap H^2 (\mathbb{T})$.
Since $\widetilde{C_R}$ is an isometry, we have that 
\[
 (\widetilde{C_R} ^* T_a \widetilde{C_R} \xi_m)(z) = \sum_{l=0} ^{\infty} c_l z^{l+m}.
\]
On the other hand  $\widetilde{C_R} ^* = \mathcal{L}_R$ by Lemma  \ref{lem:bounded}, 
hence 
\[
 \mathcal{L}_R(a) = \sum_{l=0} ^{\infty} c_l \mathcal{L}_R (R^l).
\]
By Lemma \ref{lem:L(1)=1},
\[
 \mathcal{L}_R (R^l)(w) = \frac{1}{n}\sum_{z \in R^{-1} (w)} h(z)R(z)^l 
=  (\mathcal{L}_R (1)(w))w^l = w^l.
\]
Thus
\[
 \mathcal{L}_R (a)(w) = \sum_{l=0} ^{\infty} c_l w^l  \ \ \ \text{as $L^2$-convergence}.
\]
Since $\mathcal{L}_R(a) \in H^{\infty}(\mathbb{T})$, we have
\[
 (T_ {\mathcal{L}_R(a)} \xi_m)(w) = (\mathcal{L}_R(a) \xi_m)(w)
 = (M_{\xi_m} \mathcal{L}_R (a))(w) = \sum_{l=0} ^{\infty} c_l w^{l+m}
\]
for $m \geq 0$, where $M_{\xi_m}$ is a multiplication operator by
$\xi_m$. Therefore we obtain that 
$\widetilde{C_R} ^* T_a \widetilde{C_R} = T_{\mathcal{L}_R(a)}$.

For the remaining case that $j \leq 0$, the formula 
$\widetilde{C_R} ^* T_a \widetilde{C_R} = T_{\mathcal{L}_R(a)}$
also holds because 
$\mathcal{L}_R$ is positive and $T_a ^* =
T_{\overline{a}}$. 
\end{proof}

\begin{lem} \label{lem:CONS}
Let the notation be as above. Then
\[
 \sum_{k=1} ^n T_{Q_{k-1}R_{k-1}} \widetilde{C_R} \widetilde{C_R} ^*
 (T_{Q_{k-1} R_{k-1}} )^* = I.
\]
\end{lem}

\begin{proof}
We have
\[
 T_{Q_{k-1} R_{k-1}} \widetilde{C_R} \xi_l = Q_{k-1} R_{k-1} R^l = e_{ln+(k-1)}
\]
for $1 \leq k \leq n$ and $l \geq 0$. Thus $T_{Q_{k-1} R_{k-1}} \widetilde{C_R}$ is
an isometry and the desired equality follows.
\end{proof}

We can now state our main theorem. 

\begin{thm}
Let $R$ be a finite Blaschke product of degree at least two and $R(0)=
0$, that is
\[
 R(z)=\lambda z \prod_{k=1} ^{n-1} \frac{z-z_k}{1-\overline{z_k} z},
\]
where $n \geq 2, \, |\lambda| =1$ and $| z_k | < 1$.
Then the quotient $C^*$-algebra $\mathcal{OC}_R$ of 
Toeplitz-composition $C^*$-algebra $\mathcal{TC}_R$ by 
the ideal of compact operators is 
isomorphic to the $C^*$-algebra $\mathcal{O}_R (J_R)$
associated with the complex dynamical system 
on the Julia set $J_R$ 
by an isomorphism $\Psi$ such that $\Psi(\pi(T_a)) = a$ for $a \in C(\mathbb{T})$ and
$\Psi(\pi(\widetilde{C_R})) = \sqrt{\frac{R}{z R'}} S_1$,
where $\pi$ is the canonical quotient map $\mathcal{TC}_R$ to
$\mathcal{OC}_R$ and $1$ is the constant map in $X_R$ taking constant 
value $1$. Moreover the $C^*$-algebra $\mathcal{OC}_R$ is 
simple and purely infinite. 
\end{thm}

\begin{proof}
For any $\xi \in X_{R,h}$ and $a \in A$,
let $p(z) = \xi (z, R(z)), \, \rho(a) =\pi(T_a), \,
V_{\xi} = n^{1/2} \rho(p)\pi(\widetilde{C_R})$.
We have
\[
 \rho(a) V_{\xi} = n^{1/2} \rho(ap) \pi(\widetilde{C_R}) = V_{a \cdot \xi}.
\]
Since, for $b \in A$
\[
 \pi(\widetilde{C_R})\rho(b) = \pi(\widetilde{C_R}T_b)
= \pi(T_{b \circ R} \widetilde{C_R}) = \rho(b \circ R) \pi(\widetilde{C_R}),
\]
we have  that
\begin{align*}
 V_{\xi} \rho(b) &= n^{1/2} \rho(p) \pi(\widetilde{C_R}) \rho(b)
 = n^{1/2} \rho(p) \rho(b \circ R) \pi(\widetilde{C_R}) \\
 &= n^{1/2} \rho(p(b \circ R)) \pi(\widetilde{C_R}) = V_{\xi \cdot b}.
\end{align*}
For any $ \eta \in X_{R,h}$, define $q(z) = \eta (z, R(z))$.
By Proposition \ref{prop:transfer},
\begin{align*}
 V_{\xi} ^* V_{\eta} &= n \pi(\widetilde{C_R}^*) \rho(\overline{p}) \rho(q)
 \pi(\widetilde{C_R})
 = n \pi(\widetilde{C_R} ^* T _{\overline{p}q} \widetilde{C_R}) \\
 &= n \pi(T_{\mathcal{L}_R (\overline{p} q)}) = \rho(n
 \mathcal{L}_R (\overline{p}q)) = \rho(( \xi | \eta )_{A,h}).
\end{align*}
Set $v_k (z, R(z))= Q_{k-1} (z) R_{k-1} (z)$ for $1 \leq k \leq n$.
By Lemma \ref{lem:CONS},
\[
 \sum_{k=1} ^n V_{v_k} V_{v_k} ^*
 = \sum_{k=1} ^n\pi(T_{Q_{k-1}R_{k-1}} \widetilde{C_R}
   \widetilde{C_R} ^* (T_{Q_{k-1} R_{k-1}}) ^*) = I.
\]
By the universality and the simplicity of $\mathcal{O}_{R,h} (\mathbb{T})$,
there exsists an isomorphism $\Omega : \mathcal{O}_{R,h}(\mathbb{T}) \rightarrow \mathcal{OC}_R$
such that $\Omega(\hat{S}_{\xi})= V_{\xi}$ and $\Omega(a) =\rho(a)$
for $\xi \in X_{R,h}, \, a \in A$.
Let $\Phi$ be the map in  Lemma \ref{lem:isomorphic} and put $\Psi = \Phi \circ \Omega^{-1}$.
Then $\Psi$ is the desired isomorphism. Since it is proved in 
\cite{Kajiwara and Watatani}  that 
$C^*$-algebra $\mathcal{O}_R(J_R)$ is simple and purely infinite, 
so is $C^*$-algebra $\mathcal{OC}_R$. 
The rest is now clear. 
\end{proof}

\begin{rem}
It is important to notice that 
the element $\pi(\widetilde{C_R})$ of the composition operator 
in the quotient algebra 
exactly corresponds to the implementing isometry operator in 
Exel's crossed product $A \rtimes _{\alpha, \mathcal{L}_R} \mathbb{N}$ 
in \cite{Exel}, 
which depends on the transfer opearator $\mathcal{L}_R$. 
It directly follows from the fact that  $\mathcal{O}_{R,h} (\mathbb{T})$ 
is naturally isomorphic to $A \rtimes _{\alpha, \mathcal{L}_R} \mathbb{N}$. 
\end{rem}

\begin{ex}
Let $R(z) = z^n$ for $n \geq 2$. Then 
the Hilbert bimodule $X_R$ over $A = C({\mathbb T})$ 
is isomorphic to $A^n$ as a right $A$-module. 
In fact, let $u_i(z,w) = \frac{1}{\sqrt{n}} z^{i-1}$ for 
$i = 1,\dots, n$.  Then $(u_i |u_j)_A = \delta_{i,j}I$ 
and $\{u_1, \dots, u_n\}$ is a basis of $X_R$. 
Hence $S_i := S_{u_i}$, $(i = 1,...,n)$ are generators of the 
Cuntz algebra $\mathcal{O}_n$.  We see that 
$(z \cdot u_i)(z,R(z)) = u_{i+1}(z,R(z))$ 
for $i =1,...,n-1$ and $(z \cdot u_n)(z,R(z)) = z^n = (u_1\cdot z)(z,R(z))$ and
the left multiplication by $z$ is a unitary $U$. 
Therefore  $C^*$-algebra $\mathcal{O}_R (J_R)$
associated with the complex dynamical system on the Julia set $J_R$ 
is the universal $C^*$-algebra generated by a unitary $U$ and $n$ isometries 
$S_1, ...,S_n$ satisfying $S_1S_1^* + \dots + S_nS_n^* = I$ , 
$US_i = S_{i +1}$ for $i =1,...,n-1$ and $US_n = S_1U$. 
In this way, the operator $S_1$ corresponds to the element
$\pi(\widetilde{C_R}) \in \mathcal{OC}_R$ of 
the composition operator $C_R$ and the unitary $U$  corresponds to 
the element $\pi(T_z) \in \mathcal{OC}_R$. Moreover we find that a commutation 
relation $U^nS_1 = S_1U$ in the $C^*$-algebra $\mathcal{O}_R (J_R)$
and a commutation relation $T_{R(z)}C_R  = C_RT_z$  in the 
Toelitz-composion $C^*$-algebra $\mathcal{TC}_R$ are essentially same one. 
\end{ex}

The Toeplitz-compostion $C^*$-algebra 
depends on the analytic structure of the Hardy space by the costruction. 
The finite  Blaschke product $R$ is not conjugate with $z^n$  by any 
M\"{o}bius automorphism unless $R(z) = \lambda z^n$. But we can show that the 
the quotient algebra  $\mathcal{OC} _{R}$ is isomorphic to $\mathcal{OC} _{z^n}$as a corollary of our main theorem. 

\begin{cor}
Let $R$ be a finite Blaschke product of degree $n \geq 2$ with $R(0)=0$. 
Then the quotient $C^*$-algebra $\mathcal{OC}_R$ is isomorphic to 
the $C^*$-algebra $\mathcal{OC}_{z^n}$. Moreover 
$K_0(\mathcal{OC}_R) \simeq 
{\mathbb Z} \oplus {\mathbb Z}/(n-1){\mathbb Z} $  
and $K_1(\mathcal{OC}_R) \simeq {\mathbb Z}     $. 
\end{cor}

\begin{proof}
Since 
$$
\psi '(\theta) 
  = \frac{e^{i\theta} R'(e^{i\theta})}{R(e^{i\theta})} 
= 1 + \sum_{k=1}^{n-1}
 \frac{1-|z_k|^2}{|e^{i\theta}-z_k|^2} >1,  
$$
$R|_{\mathbb T}$ is an expanding map of degree $n$ and 
$R|_{\mathbb T}$ is topological conjugate to $z^n$ on 
$\mathbb T$ by \cite{Shub}. See a general condition 
given in  Martin \cite{Martin}. Therefore the statement 
follows from the above theorem and the results in 
\cite{Kajiwara and Watatani}.

\end{proof}

\end{document}